\documentclass[10pt]{amsart}

\usepackage{amsmath, amsthm, amssymb, xcolor, graphicx, enumitem, 
everypage, datetime, array} 
\usepackage{wrapfig}

\newcounter{citedtheorems}

\newtheorem{defn}{Definition}[section]
\newtheorem{theorem}[defn]{Theorem}

\newtheorem*{theorem-m}{Theorem \ref{main-theorem}}
\newtheorem*{thm-p2a}{Theorem \ref{t:p2a}}

\newtheorem*{thm-seq}{Theorem \ref{t:seq}}
\newtheorem*{thm-e}{Theorem}
\newtheorem*{thm-m}{Main Theorem}
\newtheorem*{theorem-abs1}{Theorem \ref{ind-theorem}}
\newtheorem*{theorem-abs2}{Theorem \ref{a23}}
\newtheorem*{theorem-abs3}{Theorem \ref{ind-new}}
\newtheorem*{theorem-abs4}{Theorem \ref{m1}}
\newtheorem*{thm-x}{Theorem}
\newtheorem{thm-lit}[citedtheorems]{Theorem}
\newtheorem{defn-lit}[citedtheorems]{Definition}
\newtheorem{fact-lit}[citedtheorems]{Fact}

\newtheorem{cor}[defn]{Corollary}

\newtheorem{defn-claim}[defn]{Definition/Claim}

\newtheorem*{defn-in}{Definition \arabic{section}.\arabic{equation}}

\newtheorem*{claim-in}{Claim \arabic{section}.\arabic{equation}}

\newtheorem{concl}[defn]{Conclusion}
\newtheorem{conv}[defn]{Convention}
\newtheorem{claim}[defn]{Claim}

\newtheorem{lemma}[defn]{Lemma}
\newtheorem{obs}[defn]{Observation}
\newtheorem{rmk}[defn]{Remark}

\newtheorem{ntn}[defn]{Notation}
\newtheorem{disc}[defn]{Discussion}

\newtheorem{expl}[defn]{Example}

\newcommand{\lost}{\L os' }

\newcommand{\br}{\vspace{2mm}}

\newcommand{\mcp}{\cP}

\newcommand{\tfeq}{T_{\operatorname{feq}}}

\newcommand{\tlf}{\trianglelefteq}

\newcommand{\cf}{\operatorname{cof}}

\newcommand{\uu}{\mathcal{U}}

\newcommand{{\xw}}{\mathbf{w}}

\newcommand{\vv}{\mathcal{V}}

\newcommand{\cP}{\mathcal{P}}

\newcolumntype{L}{>{\centering\arraybackslash}m{4cm}}

\newcommand{\ii}{\iota}

\newcommand{\tp}{\operatorname{tp}}

\newcommand{\tcb}{}

\newcommand{\lgn}{\operatorname{lgn}}

\newcommand{\bz}{\mathbf{s}}

\newcommand{\mcl}{\mathcal{L}}

\newcommand{\de}{\mathcal{D}}

\newcommand{\ts}{\mathbf{S}}

\newcommand{\mch}{\mathcal{H}}

\newcommand{\rstr}{\upharpoonright}

\newcommand{\inc}{\operatorname{inc}}

\newcommand{\vp}{\varphi}
\newcommand{\lcf}{\operatorname{lcf}}

\newcommand{\trg}{T_{\mathbf{rg}}}

\newcommand{\leaf}{\operatorname{leaf}}
\newcommand{\leaves}{\operatorname{leaves}}

\title[New simple theories from hypergraph sequences]{New simple theories from hypergraph sequences}

\author{M. Malliaris and S. Shelah}
\thanks{\emph{Thanks:} Research partially supported by NSF 1553653, and by an 
NSF-BSF award (NSF 2051825, BSF 3013005232). 
Paper 1206 in Shelah's list.}

\address{Department of Mathematics, University of Chicago, 5734 S. University, Chicago, IL 60637, USA} 
\email{mem@math.uchicago.edu}

\address{Einstein Institute of Mathematics, Edmond J. Safra Campus, Givat Ram, The Hebrew
University of Jerusalem, Jerusalem, 91904, Israel, and Department of Mathematics,
Hill Center - Busch Campus, Rutgers, The State University of New Jersey, 110
Frelinghuysen Road, Piscataway, NJ 08854-8019 USA}
\email{shelah@math.huji.ac.il}
\urladdr{http://shelah.logic.at}

\begin{document}

\begin{abstract} 
We develop a family of simple rank one theories built over quite arbitrary sequences of finite hypergraphs. 
(This extends an idea from the recent proof that Keisler's order has continuum many classes, however, 
the construction does not require familiarity with the earlier proof.) 
We prove a model-completion and quantifier-elimination result for theories in this family. 
We develop a combinatorial property which they share. 
We invoke regular ultrafilters to show the strength of this property, showing that any flexible ultrafilter which is 
good for the random graph is able to saturate such theories. 
\end{abstract}

\maketitle 

\setcounter{tocdepth}{1}

\vspace{5mm}
\hfill{\emph{Dedicated to Boris Zilber on the occasion of his 75th birthday.}}

\vspace{5mm}

It is our pleasure to dedicate this to Boris for all the wonderful discoveries in model theory and its interaction with the rest of 
mainstream mathematics. 

Recently, we proved that Keisler's order has continuum many pairwise incomparable classes, within the simple rank one theories \cite{MiSh:1167}.  
A surprising point of that proof is that the theories built to obtain the continuum many incomparable classes can be 
very well understood, and are close to the random graph in various precise ways.  So we can analyze carefully how their types are 
realized and omitted; this understanding helps in proving incomparability.  
Briefly, those theories were built over template sequences of growing finite graphs, and aspects of the 
combinatorics of the template graphs such as edge densities played a role in the behavior of types in the associated theories.  
This was a very nice interaction of the finite and the infinite, where the role of graphs seemed central; 
we should ask whether this understanding 
applies to a larger, significant family of simple theories. 

In the present paper, we indeed find a way to extend ideas from construction of the theories in \cite{MiSh:1167} to build a 
nontrivial family of theories close to the random graph.   Informally, the previous idea of using  
templates of sequences of growing finite graphs can be extended to templates of 
sequences of growing finite hypergraphs of any arity.  We also indicate modifications of the construction 
involving equivalence relations rather than trees.   Although we have found these theories in the context of investigating 
Keisler's order, indications are that they may be of general interest.  Hence we have taken care to present them in the 
present paper in a hopefully easily accessible way. 

Meanwhile, an interesting aspect of Keisler's order on simple unstable theories is that it seems to be pointing the way towards isolating and 
analysing an interesting family of theories ``near'' the random graph, which includes the incomparable theories of \cite{MiSh:1167}, 
and now the more general family developed here.  We do not yet have indications whether this is \emph{the} family. 
We do intend to look at whether the incomparability via ultraproducts can be carried out at the generality of these theories, and to 
consider other related questions in a future manuscript. 

We thank the anonymous referee for thoughtful comments on the manuscript. 

\tableofcontents

\section{Templates and theories}  \label{t:theories}

To define our theories we will first need to define a template, which is a growing sequence of finite hypergraphs, all of the 
same fixed arity $k$, 
satisfying certain mild conditions the number of nodes and of edges.  Our main case is $k > 2$, but the construction also makes sense for
$k=2$ (graphs) and so generalizes a slight variant\footnote{The reader familiar with the earlier paper will remember that the 
theories there were built on bipartite graphs, which had certain advantages for the ultrapower analysis.  
In order to extend to hypergraphs, rather than solving the problem of extending the bipartition to 
a multi-partition, the problem was solved in a more satisfying way by eliminating the bi-partition; then the extension to 
higher arities is even more natural.} of the construction from \cite{MiSh:1167}.  The construction a priori makes sense without 
the conditions in \ref{d:template}, but the model completion and quantifier elimination arguments use them. 
Given any such template, we then build a theory in a natural way. 

\begin{defn}
Given a hypergraph $(H,E)$ where $E$ is a relation of arity $k$, say that $k$ is the \emph{arity} of the hypergraph. 
\end{defn}

\begin{defn} \label{d:k-full}
Call a hypergraph $(H, E)$ of arity $k$ a \emph{$k$-full hypergraph} if we can partition $E = E^* \cup E^{<k}$ such that 
$(H,E^*)$ is a $k$-uniform hypergraph, meaning the edge relation is symmetric and irreflexive and holds only on 
tuples of $k$ distinct elements, and $E^{<k}$ holds on \emph{all} tuples with $<k$ distinct elements.  
\end{defn}

Informally, $k$-full hypergraphs are those obtained by 
starting with a $k$-uniform hypergraph, 
where the edge is symmetric and irreflexive and holds only on tuples of $k$ distinct elements, 
and then extending it by setting the edge relation to hold on \emph{all} tuples with repetition. 
(This is a technical help since non-edges in template hypergraphs will indicate inconsistency in the related theory.) 
Note that it still is well defined to call $k$ the arity of the hypergraph. 

\begin{defn}
Given a hypergraph $(H, E)$ of arity $k$, 
a \emph{$k$-full-clique} is a set\footnote{In the interesting case, a set with $\geq k$ members, but 
this hypothesis is not strictly needed as the sequences can contain repetitions. In the case of the independent set, 
we need $|A| \geq k$ and could have asked $|A| > k$.} $A \subseteq H$ where every sequence of $k$ elements of 
$A$ belongs to $E$, and a 
\emph{$k$-independent set} is a set $A \subseteq H$ with $\geq k$ members 
such that no sequence of $k$ distinct elements of $A$ belongs to $E$. 
\end{defn}

\begin{defn} \label{d:template}
A \emph{template} of arity $k$, $2 \leq k < \omega$, consists of a
sequence $\mch = \bar{\mathbf{h}} = \langle \mathbf{h}_n : n < \omega \rangle$ 
and a function $f_\mch: \omega \rightarrow \omega \setminus \{ 0 \}$
such that:
\begin{enumerate}
\item[(0)] 
$\lim_{n \rightarrow \infty} f_\mch(n) = \infty$, meaning that for every $N <\omega $ there is $n < \omega$ such that 
$m \geq n \implies f_\mch(m) \geq N$. 
\item  for all $n<\omega$, 
$\mathbf{h}_n = (H_n, E_n)$ is a finite $k$-full hypergraph, $H_n = ||\mathbf{h}_n|| $ is a finite cardinal and so 
we identify the set of vertices $H_n$ with the set $\{ 0, \dots, H_n-1 \}$.   
\end{enumerate}
Moreover, for all $n<\omega$:

\begin{enumerate}[resume]
\item $f_\mch(n) \leq H_n  < \aleph_0$.
\item \emph{(Extension)}  Let $t = f_\mch(n)$. 
For every $i^0_0, \dots, i^{0}_{k-2}, 
\dots, i^{t-1}_{0},\dots, i^{t-1}_{k-2}$ from $H_n$, there exists $s \in H_n$ such that 
$\langle s, i^\ell_0,\dots, i^\ell_{k-2} \rangle \in E_n$ for all $\ell<t$. 
\end{enumerate}
We say $\mch$ is a template if $(\mch, f)$ is for some $f$. 
\end{defn}

\begin{rmk}
For notational simplicity in Definition $\ref{d:template}$, we fix $k$. We could also have defined a parameter 
$\mathbf{k}_n$ for each $n$ measuring the fullness. 
\end{rmk}

\begin{defn}
A \emph{template} is a template of arity $k$ for some $k<\omega$. 
\end{defn}

For example, the sequence of hypergraphs given by $H_n = n+1$ and $E_n = {^k H_n}$ is a template of arity $k$. 
As a more interesting example, choose the $\mathbf{h}_n$s to be a sequence of finite random hypergraphs, 
with size and edge probability sufficient to give the extension condition \ref{d:template}(3).  For a similar sufficient calculation in the 
original case of graphs, see \cite{MiSh:1167} \S 6. 

As the next definition suggests, it will be useful to think of trees naturally associated to paths through the template hypergraphs. 

\begin{defn} Given a template $\mch$, and recalling $H_n$ from $\ref{d:template}$, 
define
\[ X_\mch := \{ \rho: 
\rho \in {^{\omega >} \omega}, 0 \leq \rho(n) < H_n \mbox{ for all } n<\lgn(\rho) \} ~\} \]
to be, informally, the set of finite sequences of choices of vertices from initial segments of our hypergraph sequence, 
naturally forming a tree. 
Define 
\[ \leaves(X_\mch) = \{ \rho \in {^\omega \omega} ~:~ \eta_i \rstr n \in X_\mch \mbox{ for all } n<\omega \} \]
to be the ``limit points'' of this set.
\end{defn}

\begin{defn} \label{d:t-zero}
We define a theory $T_0=T_0(\mch)$ based on the template $\mch$ to be the following 
universal theory 
in the following language.
\begin{enumerate}
\item $\mcl = \mcl_{\mch}$ contains equality, a $k$-place relation $R$, and countably many unary predicates 
\[ \{ Q_\eta : \eta \in X_\mch \}. \]
\item $T_0$ contains universal axioms stating: 
$R$ is a symmetric $k$-uniform hypergraph, i.e. $R$ holds only on distinct $k$-tuples and if it holds on 
some $k$-tuple it holds on all its permutations.\footnote{Note that $R$, the hypergraph relation in the theory, is 
symmetric irreflexive, while the $E_n$s, the hypergraph relations in the templates, need not be irreflexive 
by the definition of ``$k$-full.''}  
\item If $\eta \tlf \nu \in X_\mch$ then $T_0$ contains the axiom: 
\[ (\forall x) (Q_{\langle \rangle} (x)) ~\land ~ (\forall x) (Q_\nu(x) \implies Q_\eta(x)) \]
saying that $Q_{\langle \rangle}$ names everything, and $Q_\nu$ refines $Q_\eta$.
\item If $\eta \in X_\mch$, $\lgn(\eta) = m$ and $i \neq j < ||\mathbf{h}_m||$ then $T_0$ contains the axiom:
\[ (\forall x) \left(\neg (Q_{\eta^\smallfrown \langle i \rangle } (x) \land  Q_{\eta^\smallfrown \langle j \rangle} (x))\right) \]
moreover, $T_0$ contains the axiom $(\forall x)( Q_\eta(x) \implies \bigvee_i Q_{\eta^\smallfrown \langle i \rangle}(x))$, 
so the predicates $\langle Q_{\eta^\smallfrown \langle i \rangle} : i <  ||\mathbf{h}_m|| \rangle$ partition $Q_\eta$. 
\item For every $\eta_0, \dots, \eta_{k-1}$ from $X_\mch$ and $n < \min \{ \lgn(\eta_0), \dots, \lgn(\eta_{k-1}) \}$, if 
$\langle \eta_0(n), \dots, \eta_{k-1}(n) \rangle \notin E_n$ then $T_0$ contains the axiom 
\[ \forall x_0, \dots, x_{k-1} \left( P_{\eta_0}(x_0) \land \cdots \land P_{\eta_{k-1}}(x_{k-1}) \implies \neg R(x_0, \dots, x_{k-1})   \right) \]
forbidding any edges across these predicates. 
\end{enumerate}
\end{defn}

\begin{disc} \label{d:05} 
\emph{Informally, the unary predicates give a model $M \models T_0$ the (hard-coded) structure of a tree. We have $\forall{x} Q_{\langle \rangle}(x)$. 
The model is first partitioned into predicates $Q_{\langle i \rangle}$ for $i < ||\mathbf{h}_0||$.  By induction on $m \geq 1$,   
each predicate $Q_\eta$ [where $\lgn(\eta) = m$, \emph{i.e.} $\eta$ is a function with domain $\{ 0, \dots, m-1 \}$] 
is partitioned into $||\mathbf{h}_m||$ disjoint pieces,
the $Q_{\eta^\smallfrown \langle i \rangle}$'s. 
So any $a \in M$ will be in some concentric sequence of predicates
$\langle Q_{\rho \rstr n} : \rho \in \leaves(X_\mch), n < \omega \rangle$.   Call $\rho$ the \emph{leaf} of $a$, see \ref{d:leaves}. 
 Note that we have arranged our indexing so that, 
in this notation, if $\rho(n) = i$ we have 
\[ a \in Q_{ (\rho \rstr n) ^\smallfrown \langle i \rangle } \] 
in other words, that its predicate at level $n$ corresponds to the $i$-th element of $H_n$. 
The final condition on edges amounts to the following. 
Given $a_0, \dots, a_{k-1}$ in a model $M \models T_0$, each element $a_i$ will belong to some leaf $\rho_i$, 
and an edge $R$ cannot occur on $\langle a_0, \dots, a_{k-1} \rangle$ \emph{unless} for every $n < \omega$, 
$\langle \rho_i(n) : i < k \rangle$ is an edge in $E_n$. (Since $T_0$ is a universal theory, of course, it records 
here just what is 
forbidden, and remains agnostic about whether edges do occur if permitted; a model completion, such as we 
shall construct soon, would have 
more information.)  }

\emph{Notice the `sparsification' of edges, or rather the accumulation of rules forbidding edges, 
as we go deeper into the ``tree''. If $\eta_0, \dots, \eta_{k-1}$ are elements of 
$X_\mch$ of length $m+1$, and $\langle \eta_0(m), \dots, \eta_{k-1}(m) \rangle \notin E_m$, then in $M$ we know 
there can be no $R$-edges spanning elements chosen from the predicates $Q_{\eta_0}, \dots, Q_{\eta_{k-1}}$ regardless of 
how these elements sit in subsequent predicates.  If on the other hand 
$\langle \eta_0(\ell), \dots, \eta_{k-1}(\ell) \rangle \in E_\ell$ for $\ell \leq m$, then a priori there may be edges 
spanning \emph{some} elements from the predicates $Q_{\eta_0}, \dots, Q_{\eta_{k-1}}$, but it may depend a priori  
on how those elements sit in subsequent predicates and what the templates say there.  } 
\end{disc}

The following auxiliary objects may clarify the picture. 

\begin{defn}   \label{d:leaves}  Fix a template $\mch$ of arity $k$. 
Let $T_0 = T_0(\mch)$ and let $M \models T_0$. 
\begin{enumerate}
\item For $a \in M$, define $\leaf(a)$ to be the unique $\rho \in \leaves(X_\mch)$ such that 
\[  M \models a \in Q_{\rho \rstr n}  \mbox{ for all $n < \omega$ }. \] 
\item Let $\mathbf{h}_\infty$ be the $k$-uniform hypergraph with vertex set $H_\infty := \leaves(X_\mch)$ and 
with edge relation $E_\infty$ given by:
\[ \langle \rho_0, \dots, \rho_{k-1} \rangle \in E_\infty \iff  \langle \rho_0(n), \dots, \rho_{k-1}(n) \rangle 
\in E_n \mbox{ for all $n<\omega$.} \]
Of course $\mathbf{h}_\infty = \mathbf{h}_\infty(\mch)$. 
\end{enumerate}
\end{defn}

\begin{obs}
It follows from $\ref{d:t-zero}(5)$ that if $M \models T_0$, $a_0, \dots, a_{k-1} \in M$, we can have 
$M \models R(a_0, \dots, a_{k-1})$ only if $\langle \leaf(a_0), \dots, \leaf(a_{k-1}) \rangle \in E_\infty$. 
\end{obs}

\begin{expl}
Suppose $k=3$,  $\langle 0, 1, 2 \rangle \in E_0$ and $\langle 3,4,5 \rangle \in E_1$. 
Then $R$-edges are not a priori forbidden in $T_0$ between $Q_{\langle 0 3\rangle}, Q_{\langle 1 4\rangle}, 
Q_{\langle 2 5\rangle}$, nor between $Q_{\langle 0 4\rangle}, Q_{\langle 1 5\rangle}, 
Q_{\langle 2 3\rangle}$ remembering symmetry of $E_1$, nor between $Q_{\langle 0 1\rangle}, Q_{\langle 1 0\rangle}, 
Q_{\langle 2 0\rangle}$ remembering $E_1$ is $k$-full.
\end{expl}

\vspace{5mm}

\section{Model completion and quantifier elimination}

\begin{conv} For the entirety of this section, fix a template $\mch$, $f_\mch$ of arity $k \geq 2$, \underline{thus} $\textbf{h}_\infty$ and  
$T_0$ as in $\ref{d:leaves}$ and $\ref{d:t-zero}$, respectively. 
\end{conv}

\begin{claim} 
For any $\rho \in H_\infty$ there are continuum many tuples 
$\rho_0, \dots, \rho_{k-2} \in H_\infty$ such that $\langle \rho, \rho_0, \dots, \rho_{k-2} \rangle \in E_\infty$, 
i.e. each leaf in this graph is contained in continuum many edges.  
\end{claim}

\begin{proof} 
By Extension, \ref{d:template}(3).
\end{proof}

The next claim is a key use of \ref{d:template}(3): in some sense, it shows that consistency in the template at 
large enough finite levels can be extended to full consistency.

\begin{claim}[Completion to a type]  \label{c:completion}
For any $1 \leq t<\omega$ and any choice of $t$ $k$-tuples $\rho^0_0, \dots, \rho^0_{k-2}, \dots, \rho^{t-1}_0, \dots, \rho^{t-1}_{k-2}$ 
from $H_\infty$, 
\underline{\emph{if}} there exists $\nu \in X_\mch$ such that $\lgn(\nu) > \min \{ m : f(n) \geq t $ for all $n \geq m \}$ and  
\[ \langle \nu(\ell), \rho^i_0(\ell), \dots, \rho^i_{k-2}(\ell) \rangle \in E_\ell  \mbox{ for all $i<m$ and $\ell < {\lgn(\nu)}$} \]
\emph{then} we can choose $\nu_*$ such that $\nu \tlf \nu_* \in H_\infty$ and 
\[ \langle \nu_*(\ell), \rho^i_0(\ell), \dots, \rho^i_{k-2}(\ell) \rangle \in E_\infty \mbox{ for all $i<m$ and $\ell < \omega$}. \]
\end{claim}

\begin{proof}
Let $n := \lgn(\nu)$. 
By induction on $r < \omega$ let us prove that we can find $\nu_r \in X_\mch$ of length $n+r$ 
such that $\nu \tlf \nu_r$ and 
\[ \langle \nu_r(\ell), \rho^i_0(\ell), \dots, \rho^i_{k-2}(\ell) \rangle \in E_\ell \mbox{ for all $i<m$ and $\ell < n+r$}. \]
For $\ell = 0$ take $\nu_t = \nu$. 
For $\ell > 0$, apply Extension, \ref{d:template}(5), to the tuples 
\[ \rho^0_0(n+r-1), \dots, \rho^0_{k-2}(n+r-1), \dots, \rho^{t-1}_0(n+r-1), \dots, \rho^{t-1}_{k-2}(n+r-1) \] 
in the hypergraph $\textbf{h}_{n+r-1}$ and let $b$ be the appropriate element of $H_{n+r-1}$ returned by that axiom. 
Then $\nu_{r} := {\nu_{r-1}}^\smallfrown \langle a \rangle$ fits the bill. 
\end{proof}

\br

\begin{defn} \label{d:t-zero-m}
For any $m < \omega$, define
$T^m_0$ to be the restriction of $T_0$ to the language with 
equality, a $k$-place relation $R$, and unary predicates 
\[ \{ Q_\eta : \eta \in X_\mch, \lgn(\eta) \leq m \}. \]
\end{defn}

\begin{claim}
For each $m<\omega$, the model completion $T^m$ of $T^m_0$ exists. 
\end{claim}

\begin{proof}  Just as in the case of graphs (\cite{MiSh:1167} 2.16), each $T^m_0$ is a universal theory in a finite 
relational language. The class of its models has the joint embedding property JEP for any two 
$M_1, M_2$ with $|M_1| \cap |M_2| = \emptyset$, and the amalgamation property 
AP when we have models $M_1, M_2$ and $M_0$ with $M_0 \models T^m_0$ and 
$M_0 \subseteq M_\ell$ for $\ell = 1,2$ and $|M_1| \cap |M_2| = |M_0|$.  To see this in both cases, 
the model $N$ whose domain is $|M_1| \cup |M_2|$, such that for each unary predicate $Q$, 
$Q^N = Q^{M_1} \cup Q^{M_2}$ and for the edge relation $R$, $R^N = R^{M_1} \cup R^{M_2}$ will be a model of 
$T^m_0$. Thus $T^m$ exists. 
\end{proof}

\begin{rmk}  Regarding the model completion: 
if $M \models T^m$, then $M$ is infinite, and indeed for each unary predicate $Q \in \tau(T^m)$, $Q^M$ is infinite. 
Moreover,\footnote{We can extend case (a) to $\{ \eta_0, \dots, \eta_{\ell-1} \}$ for some larger finite $\ell$ which form a $k$-full-clique 
in the same strong hereditary sense.}
 for any $\eta_0, \dots, \eta_{k-1} \in X_\mch$ with $\lgn(\eta_\ell) =m$ for $\ell < k$:
\begin{itemize}
\item[(a)] if $(\eta_0(i), \dots, \eta_{k-1}(i)) \in E_i$ for all $i<m$, then $R^M$ on 
$Q^M_{\eta_0} \times \cdots \times Q^M_{\eta_{k-1}}$ ``is a random hypergraph'' in the sense of first-order logic, 
meaning that \emph{if} 
$A \subseteq Q^M_{\eta_1} \times \cdots \times Q^M_{\eta_{k-1}}$ and 
$B \subseteq |M|^k$ and\footnote{We could have asked that 
$B \subseteq Q^M_{\eta_1} \times \cdots \times Q^M_{\eta_{k-1}}$,  
but the stronger statement is true.}
``$A \cap B = \emptyset$'' in the strong sense that no permutation of any $(a_1, \dots, a_{k-1}) \in A$ belongs to $B$, 
\emph{then} the set of formulas 
\begin{align*} p(x) = & \{ R(x,a_1, \dots, a_{k-1}) : (a_1, \dots, a_{k-1}) \in A \} \\
& \cup \{ \neg R(a,b_1,\dots, b_{k-1}) : (b_1,\dots, b_{k-1}) \in B \} 
\end{align*}
is a partial type in $M$, so in particular is realized if $A$, $B$ are both finite.\footnote{Note that 
by our assumption of the template hypergraphs being ``$k$-full,''  we are in case (a) whenever $| \{ \eta_0, \dots \eta_{k-1} \} | < k$. 
The hypergraph edge $R$ is a $k$-uniform hypergraph in $M$, of course, so any $(a_0, \dots, a_{k-1}) \in R^M$ 
will be a tuple of distinct elements, but fullness of the template hypergraphs means some of the elements in such a tuple are a priori 
allowed to come from the same predicate at any given level.  In particular, for each $\eta \in X_\mch$ with $\lgn(\eta) = m$, 
$(Q^M_\eta, R^M \rstr Q^M_\eta)$ is a random $k$-ary hypergraph in the usual sense of first-order logic.}

\item[(b)]  if not, $(\eta_0, \dots, \eta_{k-1})$,  then $(Q^M_{\eta_0} \times \cdots \times Q^M_{\eta_{k-1}})~ \cap ~R^M = \emptyset$. 
\end{itemize}
\end{rmk}

Next we upgrade \cite{MiSh:1167}, Claim 2.17 to the context of hypergraphs. 

\begin{ntn}
For an ordered set $X$, 
let $\inc_\ell(X)$ be the set of strictly increasing $\ell$-tuples of elements of $X$. 
\end{ntn}

\begin{defn} Given $T^m$ for some $m<\omega$ and $M, N \models T^m$, recall that: 
\begin{enumerate}
\item $a_1, \dots, a_n \in |M|$ and $b_1, \dots, b_n \in  |N|$ have the same quantifier-free 
$\tau(T^m)$-type when they agree on equality, instances of $R$, and predicates $Q_\eta$ up to $\lgn(\eta) = m$. 

\item $\vp(x,y_1,\dots, y_n)$ is a \emph{complete} quantifier-free formula of $\tau(T^m)$ 
when: 
\begin{enumerate}
\item  for every unary predicate $Q \in \tau(T^m)$ and variable $z \in \{ x, y_1, \dots, y_n \}$, either $\vp \vdash 
Q(z)$ or $\vp \vdash \neg Q(z)$, 
\item  for every $z_0, z_1$ from $\{ x, y_1, \dots, y_n \}$ either 
$\vp \vdash z_0 = z_1$ or $\vp \vdash z_0 \neq z_1$,   
\item for every $z_0, \dots, z_{k-1}$ from $\{ x, y_1, \dots, y_n \}$, 
either $\vp \vdash R(z_0, \dots, z_{k-1})$ or $\vp \vdash \neg R(z_0, \dots, z_{k-1})$. 
\end{enumerate}
Recall that the language is finite so this is well defined. 
\end{enumerate}
\end{defn}

\noindent Our next lemma says that for each $m$, the truth of sentences of $\tau(T^m_0)$ of length $\leq m$ soon stabilizes 
in the sequence of theories $T^{k}$ as $k$ goes to infinity. 

\begin{lemma}  \label{main-qe-lemma}
For every $m<\omega$, the following holds. 
Let 
\[ m_* \geq \min \{ n : n^\prime \geq n \implies f_\mch(n^\prime) \geq m \}. \] 
 If $M \models T^{m_*}$, $N \models T^{m_*+1}$ and $\vp$ is a sentence of $\tau(T^{m})$ 
of length $\leq m$, 
then $M \models \vp ~ \iff ~ N \models \vp$. 
\end{lemma}

\begin{proof}
To prove the Lemma by induction on complexity of formulas, it suffices to show:

\begin{quotation} 
\noindent $(\star)$  
Suppose $\vp(x,y_1,\dots, y_n)$ is a {complete} quantifier-free formula of $\tau(T^m)$ of length $\leq m$, so note $n<m$.  
Suppose $a_1, \dots, a_n \in |M|$ and $b_1, \dots, b_n \in  |N|$ have the same quantifier-free 
$\tau(T^m)$-type. 

Then there exists $a \in |M|$ such that $M \models \vp(a, a_1, \dots, a_n)$ \underline{if and only if} there exists $b \in |N|$ such that 
$N \models \vp(b, b_1, \dots, b_n)$. 
\end{quotation}
Without loss of generality, the sequences $a_1, \dots, a_n$ and $b_1, \dots, b_n$ are without repetition. 

For left to right, suppose that $a \in |M|$ exists and $a \in \{ a_1, \dots, a_n \}$, otherwise it is trivial.  
We will need notation to record edges and non-edges made by $a$. 
For $\bar{\ii}$ any sequence of elements of $\{ 1, \dots, n \}$, denote by $\bar{a}_{\bar{\ii}}$ the sequence 
$\langle a_{\bar{\ii}(\ell)} : \ell < \lgn(\bar{\ii}) \rangle$. 
Let 
\[ C = \{ \bar{\ii} = \langle i_{0}, \dots, i_{k-2} \rangle : \bar{\ii} \in \inc_{k-1}( \{ 1, \dots, n \} ), ~\langle a \rangle^\smallfrown\bar{a}_{\bar{\ii}} \in R^M~ \} \]
represent the set of $R$-edges made by $a$ to $\{ a_1, \dots, a_{n-1} \}$.  
Note that $|C| < n^k$. Correspondingly, let 
\[ D = \inc_{k-1}( \{ 1, \dots, n \} ) \setminus C \]
represent the set of non-$R$-edges made by $a$ to $\{ a_1, \dots, a_{n-1} \}$.  If $C = \emptyset$ finding a corresponding 
$b$ is immediate, so assume $C \neq \emptyset$. 

Each element $c$ of $M$ belongs to a unique predicate $Q_\eta$ with $\lgn(\eta) = m_*$; call it ``the $m_*$-leaf of $c$'' 
and write $\leaf_{m_*}(c) = \eta$.   Let $\rho = \leaf_{m_*}(a)$ and let $\rho_i = \leaf_{m_*}(a_i)$ for $i = 1,\dots, n$. 
 The definition of $T^{m_*}_0$ and the existence of $a$ tell us that necessarily 
\begin{quotation}
\noindent for every $\ii = \langle i_0, \dots, i_{k-2} \rangle \in C$, for every $\ell < {m_*}$, 
\[ \langle \rho(\ell), \rho_{i_0}(\ell), \dots, \rho_{i_{k-2}}(\ell) \rangle \in E_\ell. \]
\end{quotation}
Meanwhile each element $d$ of $N$ belongs to a unique predicate $Q_\eta$ with $\lgn(\eta) = {m_*}+1$; write 
$\leaf_{{m_*}+1}(d) = \eta$. So let $\nu_i = \leaf_{{m_*}+1}(b_i)$ for $i = 1, \dots, n$. Note that $\leaf_{m_*}$ and $\leaf_{{m_*}+1}$ 
a priori depend on the models $M$ and $N$, but by our assumption that $a_1,\dots, a_n$ and $b_1,\dots, b_n$  
have the same quantifier-free $\tau(T_m)$-type, necessarily $\nu_i \rstr {m_*} = \rho_i$ for $i= 1, \dots, n$. 
Apply Extension \ref{d:template}(3) to the set of $(k-1)$-tuples 
\[  \{   \langle \nu_{i_0}(\ell), \dots, \nu_{i_{k-2}}(\ell)  \rangle : \bar{\ii} = \langle i_0, \dots, i_{k-2} \rangle \in C  \} \]
recalling our choice of $m_*$, 
and let $s$ be the element of $H_{m_*}$ returned. Define $\nu = \eta^\smallfrown \langle s \rangle$. 
Now we have that  
\begin{quotation}
\noindent for every $\ii = \langle i_0, \dots, i_{k-2} \rangle \in C$, for every $\ell < {m_*}+1$, 
\[ \langle \nu(\ell), \nu_{i_0}(\ell), \dots, \nu_{i_{k-2}}(\ell) \rangle \in E_\ell. \]
\end{quotation}
So by definition of $T^{{m_*}+1}_0$,  
$\vp(x,b_1,\dots, b_n)$ is consistent with $N$, and $b$ exists because $N$ is model complete. 

The other direction, right to left, is simpler: suppose that $b \in |N|$ exists and $b \notin \{ b_1, \dots, b_n \}$. 
As before, define $C$ to be the set of representatives of edges.  
Suppose $\leaf_{{m_*}+1}(b) = \nu$ and $\leaf_{{m_*}+1}(b_i) = \nu_i$. Then since $b$ exists and $N$ 
is a model of $T^{{m_*}+1}_0$, necessarily 
\begin{quotation}
\noindent for every $\ii = \langle i_0, \dots, i_{k-2} \rangle \in C$, for every $\ell < {m_*}+1$, 
\[ \langle \nu(\ell), \nu_{i_0}(\ell), \dots, \nu_{i_{k-2}}(\ell) \rangle \in E_\ell. \]
\end{quotation}
A fortiori, then, 
\begin{quotation}
\noindent for every $\ii = \langle i_0, \dots, i_{k-2} \rangle \in C$, for every $\ell < {m_*}$, 
\[ \langle \nu(\ell), \nu_{i_0}(\ell), \dots, \nu_{i_{k-2}}(\ell) \rangle \in E_\ell \]
\end{quotation}
so by definition of $T^{m_*}_0$, $\vp(x,a_1,\dots, a_n)$ is consistent with $M$, and since it is complete 
$\vp \vdash Q_{(\nu \rstr {m_*})}(x)$,  and $a$ exists because $M$ is model complete. 
\end{proof}

\begin{cor}
``The limit theory of $\langle T^m : m < \omega \rangle$ is well defined and is a complete, model complete theory which extends $T_0$.''    
For every $m < \omega$ and every formula $\vp$ of $\tau(T^m)$ in at least one free variable,\footnote{since we do not
have constants in the language} 
for some quantifier-free formula $\psi$ of $\tau(T^m)$, for every 
$n$ large enough, we have that 
\[ (\forall \bar{x}) \left(   ~\vp(\bar{x}) \equiv \psi(\bar{x})~   \right) \in T^n. \]
\end{cor}

\begin{lemma} \label{t:rank1}
The theory $T$ is simple rank $1$.
\end{lemma}

\begin{proof}
Assume for a contradiction that $\langle \bar{a}_i : i < \kappa \rangle$, $\kappa= \cf(\kappa) \geq (2^{\aleph_0})^+$ witnesses that some 
formula $\vp(\bar{x}, \bar{y})$ $n$-divides, in a large $\kappa$-saturated model $M \models T$.  
Without loss of generality, possibly adding dummy variables, 
$\lgn(\bar{x}) = \lgn(\bar{y}) =: m$.  

For each $i < \kappa$, let $\bar{b}_i$ be such that $M \models \vp[\bar{b}_i, \bar{a}_i]$. 
Since $\kappa$ is large enough (i.e., since $\cf(\kappa) > 2^{\aleph_0}$), for some $\uu \in [\kappa]^\kappa$, for each $\ell < m$ there is 
$\nu_\ell  \in H_\infty$ such that $\leaf(\bar{b}_{i,\ell})$ is constantly equal to $\nu_\ell$, and there is    
$\rho_\ell \in H_\infty$ such that $\leaf(\bar{a}_{i,\ell})$ is constantly equal to $\rho_\ell$.   

Let $\vp^\prime$ be an extension of $\vp$ which is complete for $\{ = , R \}$  (it will obviously only contain 
information about unary predicates up to some finite level)
such that $M \models \vp^\prime[\bar{b}_i, \bar{a}_i]$ 
for $i \in \vv \in [\uu]^\kappa$.  We may assume $\vp^\prime$ is quantifier-free. 
Without loss of generality, $\vp^\prime$ does not imply any instances of equality among the $x$'s or between the 
$x$'s and the $y$'s.  In what follows, replace $\vp$ by $\vp^\prime$ and $\langle \bar{a}_i : i < \kappa \rangle$ by 
$\langle \bar{a}_i : i \in \vv \rangle$.
 
We would like to show that 
\[ \Sigma(\bar{x}) = \{ \vp(\bar{x}, \bar{a}_i) : i < \kappa \} \mbox{ is consistent}. \] 
It suffices by induction on $j < m$ to choose elements $b_j$ so that 
$b_j$ realizes the set of formulas $\Sigma^j(b_0, \dots, b_{j-1}, x_j)$ where $\Sigma^j$ is the restriction of $\Sigma$ to the variables 
$x_0, \dots, x_j$. In the case $\ell(\bar{x})=1$, write $\nu = \leaf(x)$, and this case follows from three simple observations:
\begin{itemize}
\item $\vp$ is without loss of generality quantifier-free; 
we assumed no instances of equality between the $x$'s, and our theory has no algebraicity.  
\item the template hypergraphs contribute no restriction to this set of formulas, since if 
$R(x,a_{j_0}, \dots, a_{j_{k-2}})$ is implied by $\Sigma$ then we know by our construction that 
$(\nu, \rho_{j_0}, \dots, \rho_{j_{k-2}}) \in E_\infty$. 
\item the indiscernibility of $\langle \bar{a}_i : i < \kappa \rangle$, transitivity of equality, and consistency of each instance  
$\vp(x, \bar{a}_i)$ together 
mean that if $R(x,a_{j_0}, \dots, a_{j_{k-2}})$ is implied by $\Sigma$ and $\neg R(x,a_{\ell_0}, \dots, a_{\ell_{k-2}})$ is implied by $\Sigma$, 
then no permutation of $\langle a_{j_0}, \dots, a_{j_{k-2}}\rangle$ is equal to $\langle a_{\ell_0}, \dots, a_{\ell_{k-2}} \rangle$ (so the `positive' and `negative' edges required by $\Sigma$ cause no explicit contradiction). 
\end{itemize} 
Observe that the inductive step, since we will have already chosen the earlier values $b_\ell$ $(\ell < j$), 
will reduce to the case $\lgn(\bar{x}) = 1$ (using $\lgn(\bar{y}) = m+j$).  
This is enough to deduce the consistency of $\Sigma$, so there is no dividing. 
\end{proof} 

\begin{concl}
Given any template $\mch$, the universal theory $T_0 = T_0(\mch)$ has a model completion $T = T(\mch)$ which is 
well defined, eliminates quantifiers, is simple rank $1$, and is equal to the limit of $\langle T^m : m < \omega \rangle$. 
\end{concl}

\begin{disc} \label{c:eq}
We could have defined the theory to be ``based on'' predicates naming classes of crosscutting finite equivalence relations, 
rather than levels of trees, in the natural way.   Alternatively, we could make $E_n$ be a $k$-place relation on $\prod_{\ell \leq n} H_\ell$. 
\end{disc}

\vspace{3mm}

\section{A combinatorial property}

In this section we give Definition \ref{d:1167}, which is supposed to capture what is simple about the theories of  
\S \ref{t:theories}, not necessarily what is complicated about them.  In \S \ref{s:flexible} we shall use this to give a sufficient condition 
for ultrafilters to saturate such theories.  First let us motivate the property. 

Suppose, with no assumptions on $T$ or $\vp$, we have a sequence of instances of $\vp$ 
\[ \vp(\bar{x}, \bar{a}_0), \dots, \vp(\bar{x}, \bar{a}_{s-1}) \]
forming  a partial type, 
and suppose we replace each $\bar{a}_i$ by a sequence $\bar{b}_i$ having the same type over the empty set. 
(We don't ask that $\bar{a}_i$ and $\bar{a}_j$ have the same type for $i \neq j$, just that 
$\bar{a}_i$ and $\bar{b}_i$ have the same type for each $i$.) 
Then a priori, 
\[ \vp(\bar{x}, \bar{b}_0), \dots, \vp(\bar{x}, \bar{b}_{s-1}) \]
need not remain a partial type.  An example is $\vp(x; y_0, y_1) = y_0 < x < y_1$ in the theory of 
dense linear orders: any two pairs of increasing elements have the same type over the empty set, but we can 
choose the $\bar{a}$'s to be a sequence of intervals which are concentric, and the $\bar{b}$'s 
a sequence which are disjoint. Similar examples arise 
whenever we have a tuple beginning two indiscernible sequences, one which witnesses dividing of $\vp$ and one which 
does not.

An example of $(T, \vp)$ where such a substitution \emph{does} remain a partial type, for 
trivial reasons, is $T = \trg$ the theory of the random 
graph, and $\vp(x,y) = R(x,y)$, using only the positive instance.  Note that $\vp(x;y,z) = R(x,y) \land \neg R(x,z)$ would not 
work, however, since in changing from $\bar{a}$'s to $\bar{b}$'s we could introduce collisions among the parameters.  
A less trivial example will be the positive instance of the edge relation in the theories of \S \ref{t:theories}, which in fact satisfy a stronger condition, 
(as does the random graph), as we shall now see. 

That is: among the examples of $(T, \vp)$ where this \emph{does} work, we can ask just how much of each type 
we need to preserve when changing the parameters from $\bar{a}_i$'s to $\bar{b}_i$'s. 
Rather than preserving all formulas, 
perhaps it would be sufficient to enumerate some formulas of the type of each parameter in some coherent way, 
and then preserve some finite initial segment of each of these lists. 
It is reasonable that the length of the initial segment needed would depend on $s$, 
the number of instances we are dealing with.  This is essentially what the next 
definition says.\footnote{The provisional name is because it captures a key property of theories from \cite{MiSh:1167} = [MiSh:1167].}

\begin{defn}  \label{d:1167}  We say that $(T, \vp(\bar{x}, \bar{y}))$ has the \emph{pseudo-nfcp} when $T$ is countable and 
we can assign to each type 
$p \in \mcp$, where 
\[ \mcp := \{ p :  p \in  \ts_{\ell(\bar{y})}(\emptyset) \mbox{ and $p$ contains the formula $\exists \bar{x} \vp(\bar{x}, \bar{y})$} \} \] 
a function $f_p: \omega \rightarrow \omega$ 
such that: 
\begin{enumerate}
\item $($continuity$)$ for each $m<\omega$, if $f_p(m) = r$, then for some $\psi(\bar{y}) \in p$, for any other $q \in \mcp$, 
if $\psi \in q$, then $f_q(m) = r$. 

\item \emph{notation:} if $p = \tp(\bar{a}) \in \mcp$, we may write $f_{\bar{a}}$ for $f_p$. 

\item for every $s \geq 1$ there is some $n < \omega$  such that:
%$\oplus_{s,n}$ 
\\ whenever $\bar{a}_0, \dots, \bar{a}_{s-1}$, 
$\bar{b}_0, \dots, \bar{b}_{s-1}$ are sequences from $\mathfrak{C}_T$, hence each realizing types in $\mcp$, and 
\[ f_{\bar{a}_\ell} \rstr n = f_{\bar{b}_\ell} \rstr n \mbox{ for all $\ell < s$ } \]
and $\{ \vp(\bar{x},\bar{a}_\ell) : \ell < s \}$ is a partial type, then $\{ \vp(\bar{x}, \bar{b}_\ell) : \ell < s \}$ is also a partial type. 
In the proofs that follow, we will refer to this by saying 
\\ ``$(T, \vp)$ is $(s,n)$-compact.'' %or writing ``$\oplus_{s,n}$'' for short. 
\end{enumerate}
\end{defn}

\begin{disc}  \emph{ }
\begin{enumerate}
\item So Definition \ref{d:1167} is a kind of compactness demand, that is, given $(T, \vp(\bar{x}, \bar{y}))$, 
to know if $\mathfrak{C}_T \models (\exists{\bar{x}}) \bigwedge_{\ell<s} \vp(\bar{x}, \bar{b}_\ell)$ we need to 
know just finite approximations to the type of each $\bar{b}_\ell$ $($not of ${\bar{b}_0}^\smallfrown \cdots 
\bar{b}_{s-1}$!$)$ and the size of ``finite,'' represented here by $n$, depends just on $s$ $($and on $T$ and $\vp$$)$.

\item We could have defined the range of each function $f$ to be finite subsets 
of $\omega$, as would be convenient in $\ref{c:example2}$, or a more complicated set 
$($of bounded, say countable, size$)$; or we could have used $\{ 0, 1\}$. 

\item We could extend the definition to uncountable theories with more work. 
\end{enumerate}
\end{disc}

\begin{rmk} \label{r:1167} In the context of Definition $\ref{d:1167}$, when that definition is satisfied, we may define 
two functions $F$ and $G$.
\begin{itemize}
\item[(a)] define $F: \omega \rightarrow \omega$ by 
\[ s \mapsto \min \{ n < \omega : \mbox{$(T, \vp)$ is $(s,n)$-compact} %\oplus_{s,n} 
\} \]
which expresses: in order for $s$ instances to remain consistent, 
their functions $f$ must be preserved at least up to $F(s)$.  This is well defined since we assume 
the definition is satisfied. There are two cases:  
\begin{enumerate}
\item $\lim_{s \rightarrow \infty} F(s) \rightarrow \infty$. 
\item $\lim_{s \rightarrow \infty} F(s) = N < \infty$. 
\end{enumerate}
\item[(b)] define $G: \omega \rightarrow \omega \cup \{ \infty \}$ by: 
$n \mapsto \infty$ if $(T, \vp)$ is $(s,n)$-compact %$\oplus_{s,n}$ 
for all $n<\omega$, and otherwise 
\[  n \mapsto \max \{ s < \omega : (T, \vp) \mbox{ is $(s,n)$-compact }%\oplus_{s,n} 
\} \]
which expresses that if the functions $f$ are preserved up to $n$ then $G(n)$ instances can 
safely remain consistent.  
Here $(\star)$ $\lim_{n \rightarrow \infty} G(n) = \infty$, possibly attaining the limit already at some finite $n$. 
\end{itemize}
\end{rmk}

\begin{claim} \label{c:example2}
Let $T$ be one of the theories from $\S \ref{t:theories}$, built from $\mch$, $f_\mch$ of arity $k$.  Let $\vp(x,y_0, \dots, y_{k-2}) = 
R(x; y_0, \dots, y_{k-2})$.  Then $(T, \vp)$ has the pseudo-nfcp.
\end{claim}

\begin{proof}
In this context, by quantifier elimination, the set of $1$-types over the empty set are the set of ``leaves'', that is, each $1$-type is specified by 
choosing some $\eta \in \leaves(X_\mch)$ and considering $\{ Q_{\eta \rstr n} : n < \omega \}$. 

If $k = 2$, this also specifies $\mcp$. Otherwise, specifying a type $p(y_0, \dots, y_{k-2}) \in \mcp$ involves 
specifying the leaf of each $y_i$, and if two elements share the same leaf, whether they are equal. 

Consider any enumeration $\langle \psi_i : 1 \leq i < \omega \rangle$ of the predicates $Q_\eta(y)$ of $\tau(T)$.
which enumerates in nondecreasing order of $\lgn(\bar{\eta})$.
Fix also in advance an enumeration of the subsets of $(k-2) \times (k-2)$, and of the subsets of $k-2$. 
For each $p \in \mcp$, let $f(0)$ code the instances of equality among $y_0, \dots, y_{k-2}$, and for $1 \leq m < \omega$, 
let $f(m)$ code which subset of $\{ y_0, \dots, y_{k-2} \}$ has the $m$-th predicate as part of their type. 
[Alternately, we could have enumerated the predicates with different variables: $Q_0(y_0)$, $Q_0(y_1)$, \dots, and let $f$ 
take values in $\{ 0, 1 \}$.]

Now, if we preserve initial segments of $f$, we clearly hold constant the types of the parameters up to some level $k$ in our hard-coded tree. 
Lemma \ref{main-qe-lemma} tells us that $m$ exists as a function of $s$, as desired. 

Unless $\mch$ is very uncomplicated (for example, cliques all the way up) 
the theory will normally be in case $(a)(1)$ of $\ref{r:1167}$. 
\end{proof}

\vspace{3mm}

\section{A separation via flexibility}  \label{s:flexible}

\tcb{The theories built above are simple rank one (\ref{t:rank1} above), thus, they are low. In this section, we consider flexible ultrafilters,
those which Kunen called ``OK'', which are necessary to saturate any non-low theory in Keisler's order (see \cite{mm4}).}

\begin{defn}
Recall that the ultrafilter $\de$ on $I$, $|I| = \lambda$ is \emph{flexible} if it has a regularizing family below any nonstandard integer, 
that is, for every sequence of natural numbers $\langle n_i : i \in I \rangle$ such that $\prod_{i \in I} n_i /\de > \aleph_0$, 
there is $\{ X_\alpha : \alpha < \lambda \} \subseteq \de$ such that for all $i \in I$, 
\[ | \{ \alpha < \lambda : i \in X_\alpha \} | \leq n_i. \]
\end{defn}

\begin{defn}
Recall that a necessary and sufficient condition for a regular ultrafilter $\de$ on $I$, $|I| = \lambda$ to be \emph{good for the 
random graph} is that for any infinite $M$ and any $A, B \subseteq M^I/\de$ such that $|A| + |B| \leq \lambda$ and $A \cap B = \emptyset$, 
there is an internal predicate $P$ such that $A \subseteq P$ whereas $B \cap P  = \emptyset$. 
\end{defn}

\begin{theorem}  \label{t:flex}
Suppose $\de$ is a regular ultrafilter on $I$, $|I| = \lambda$ which is flexible and good for the random graph. 
Suppose $(T, \vp)$ has the pseudo-nfcp and $M \models T$. Then $M^I/\de$ is $\lambda^+$-saturated for positive $\vp$-types. 
\end{theorem}

\begin{proof}  
Let $M \models T$ and let $N = M^I/\de$. Consider a positive $\vp$-type $p(x)$, where 
$\vp  = \vp(\bar{x}, \bar{y})$. 
Enumerate the type as 
\[ \langle \vp_\alpha(\bar{x},\bar{a}_\alpha) : \alpha < \lambda \rangle. \]
Fix $i_* = \langle i_t : t \in I \rangle/\de$ a nonstandard integer 
(so that ``$\max$'' will be well defined). 
For a finite tuple $\bar{a}$ from $N$, 
\[ \mbox{let $f_{\bar{a}}$ mean  
$f_{\operatorname{tp}(\bar{a},\emptyset, N)}$} \] 
and given in addition an index $t \in I$, 
\[ \mbox{let $f_{\bar{a}[t]}$ mean $f_{\operatorname{tp}(\bar{a}[t],\emptyset, M)}$. } \]
For each $\alpha < \lambda$ and each $t \in I$ (i.e., for each formula and each index), define:
\begin{itemize}
\item $n(\alpha, t)$ to be the largest $n \leq i_t$ such that for all $\ell < m$, the type of 
$\bar{a}_{\alpha}[t]$ aligns with that of $\bar{a}_{\alpha}$ up to level $n$ as measured by $f$, that is,
\[ n(\alpha,t) := \max \{ n \leq i_t :  f_{\bar{a}_{\alpha}[t]} \rstr n = f_{\bar{a}_{\alpha}} \rstr n \}. \]
\item $s(\alpha, t) := G(n(\alpha, t))$, using the notation of $\ref{r:1167}$. 
\end{itemize}
The first is well defined since it the condition is trivially true for $0$. By \lost theorem, since the $f$'s reflect formulas 
for each $n<\omega$ and 
each $\alpha < \lambda$, 
\[ \{ t \in I : n < n(\alpha, t) \} \in \de. \]
Hence for each $\alpha < \lambda$, $n_\alpha := \prod_t n(\alpha,t) /\de$ is a nonstandard integer. 
It follows from \ref{r:1167}(b)($\star$) that for each $\alpha < \lambda$, $s_\alpha := \prod_t s(\alpha,t) /\de$ is either 
``$\infty$'' on a large set, or a nonstandard integer. 

Since $\de$ is good for the random graph, $\lcf(\omega, \de) \geq \lambda^+$, so there is a nonstandard integer 
$\bz = \langle \bz[t] : t \in I \rangle/\de$ such that 
for each $\alpha < \lambda$, $\bz < s_\alpha \mod \de$. 
Since $\de$ is flexible and $\bz$ is a nonstandard integer, 
we may choose $\{ X_\alpha : \alpha < \lambda \} \subseteq \de$ regularizing $\de$ and 
with the property that for each $t \in I$, 
\[ | \{ \alpha < \lambda :  t \in X_\alpha \} | \leq \bz[t]. \]
Define a map $d: [\lambda]^1 \rightarrow \de$ by:
\[ \{ \alpha \} \mapsto \{ t \in I : \bz[t] < s(\alpha,t) \} \cap X_\alpha. \]
That is, we assign $\alpha$ to an index set where we can be sure that the type of each 
$\bar{a}_{\alpha}[t]$ is ``correct'' up to the level needed to handle $s(\alpha,t) = G(n(\alpha,t))$ instances, thus a fortiori $\bz[t]$ instances. 
The intersection with $X_\alpha$ ensures, for each $t \in I$, the set $ U(t) :=  \{ \alpha : t \in d(~ \{ \alpha \} ~) \}$ of instances 
assigned to index $t$ has size $\leq \bz[t]$. 

Now for each $t \in I$, in the ultrapower $N$, 
\[ \{ \vp(\bar{x}, \bar{a}_\alpha) : \alpha \in U_t \} \]
is a set of no more than $\bz[t]$ positive instances of $\vp$, and by definition is a partial type. 
Also by our definition, for each $\alpha$, and in particular for each $\alpha \in U_t$, 
\[ f_{\bar{a}_\alpha} \rstr n(\alpha, t) =  f_{\bar{a}_\alpha[t]} \rstr n(\alpha, t). \]
It follows that
\[ \{ \vp(\bar{x}, \bar{a}_\alpha[t]) : \alpha \in U_t \} \]
remains a partial type in the index model $M$. So we can realize the type at each index under this distribution, 
and thus in the ultrapower $N$. 
\end{proof}

\begin{cor}
If $T$ is a theory from $\S \ref{t:theories}$, $M \models T$, and $\de$ is a regular ultrafilter on $I$, $|I| = \lambda$ which is 
flexible and good for the random graph, then $M^I/\de$ is $\lambda^+$-saturated. 
\end{cor}

\begin{proof}
We argue almost identically to \cite{MiSh:1167} Definition 4.7, Claim 4.8, Fact 5.2 and Conclusion 5.7 (changing just the 
arity of the edge relation, and eliminating the bi-partition from the case of graphs), that in 
regular ultrapowers which are good for the theory of the random graph, for $\lambda^+$-saturation it suffices to consider partial types of the form 
\[ p(x) = \{ Q_\nu(x) \} \cup \{ R(x,\bar{a}) : \bar{a} \in {^{k-2}A} \} \]
for $\lgn(\nu) < \omega$.  [Briefly, those definitions and claims note that any regular ultrapower has a certain weak saturation, 
for instance leaves are large, and instances of equality in types can be safely ignored. Now use quantifier elimination to get a simple normal form for types by specifying the leaf of 
$x$, a set of tuples it connects to, and a disjoint set of tuples it does not connect to.   Since saturation of ultrapowers 
reduces to saturation of $\vp$-types, it is sufficient to deal with only a finite amount of information on the leaf of $x$. 
Finally, since ``goodness for the random graph'' allows us to internally separate sets of size $\leq \lambda$, it suffices to handle the 
positive part of the type.] 
\end{proof}

\begin{defn}
For the purposes of the next corollaries, call a theory $T$ a pseudo-nfcp theory if there is a set $\Sigma$ of formulas of the language such 
$(a)$ $(T,\vp)$ has the pseudo-nfcp for each $\vp \in \Sigma$, and $(b)$ given any regular ultrafilter 
$\de$ over $\lambda$ and $M \models T$, whether $M^\lambda/\de$ is $\lambda^+$-saturated depends only on 
$\lambda^+$-saturation for positive $\vp$-types for $\vp \in \Sigma$. 
\end{defn}

\begin{cor} \label{above-low}
Let $T$ be a pseudo-nfcp theory, and let $\tlf$ denote Keisler's order.  
\begin{enumerate}
\item[(a)] Let $T_*$ be any non-low simple theory. Then $T \tlf T_*$.
\item[(b)] $T \tlf \tfeq$.
\end{enumerate} 
Thus if $T$ is a pseudo-nfcp theory and $T_*$ is any non-low or non-simple theory, $T \tlf T_*$. 
In particular, this is true for all the theories of $\S \ref{t:theories}$ above. 
\end{cor}

\begin{proof}
Any regular ultrafilter on $\lambda \geq \aleph_0$ which is good for some unstable theory is necessarily good for the random graph, 
as the random graph is the $\tlf$-minimum unstable theory. 
Any regular ultrafilter which is good for $\tfeq$ is flexible \cite[Lemma 8.8]{mm4}, and indeed and regular ultrafilter 
$\de$ which is good for some non-low simple theory is flexible 
\cite[Lemma 8.7]{mm4}.  The last line of the Corollary now follows from the fact that $\tfeq$ is the Keisler-minimum non-simple theory  
\cite[Theorem 13.1]{MiSh:998} 
(as $\tfeq$ is minimum among theories with $TP_2$, whereas $SOP_2$ implies maximality). 
\end{proof}

\begin{disc}
\emph{
\tcb{The current instances of incomparability in Keisler's order mostly use one of two main ideas. The first is to say on one hand,  
changing the distance in the alephs between $\lambda$ and some smaller $\mu$ (the size of a maximal antichain in a certain Boolean algebra used in building the ultrafilter) affects for which values of $k$ the theories $T_{k+1, k}$ are saturated, and on the other,  the ``canonical
simple non-low theory'' (see appendix to \cite{MiSh:1124}) requires the ultrafilter to be flexible; under large cardinal assumptions, these 
two indicators can be varied independently, see \cite{ulrich}, \cite{MiSh:1124}. 
In ZFC, this phenomenon can be scaled down to see an incomparability between the $T_{k+1, k}$'s and a certain theory based on trees, 
which is low \cite{MiSh:1140}.   A second, much larger scale of incomparability was produced in \cite{MiSh:1167}, with continuum many 
simple rank one theories, the graph precursors of the hypergraph theories built here. As this discussion suggests, and as the proofs of 
this section show, once the ultrafilter becomes flexible, the noise of any differences in the present theories 
is drowned out by the huge power of the 
regularizing families available. Do there exist incomparable simple non low theories? Is incomparability mainly visible in the absence of 
forking? }
}
\end{disc} 

We also record that, as an interesting immediate consequence of earlier arguments 
\cite{ulrich}, \cite{MiSh:1124}, the theories built in $\S 2$ are 
(assuming a large cardinal) distinguishable in Keisler's order from the theories $T_{k+1,k}$, the higher analogues of the 
triangle-free random graph from \cite{h:letter}. That is:

\begin{concl}
Assuming a supercompact cardinal, for arbitrarily large $\lambda$ and any $\ell <\omega$ there is a regular ultrafilter $\de$ on $\lambda$ 
which is flexible and good for the random graph, \emph{thus} good for theories of $\S 2$, but not good for 
$T_{k+1,k}$ for any $2 \leq k < \ell$. 
\end{concl}

\begin{proof} 
\cite{MiSh:1124} Claim 10.32 gives the existence of the needed ultrafilter and in clause (a) shows it is not good for $T_{k+1,k}$
 for $k < \ell$.  \cite{MiSh:1124} Claim 10.30 shows this ultrafilter is flexible and good for the random graph. 
 So by Theorem \ref{t:flex} above it can handle the theories of \S 2. 
\end{proof}

\end{document}